%% file: G2_invariant.tex
\documentclass[11pt]{amsart}
\usepackage{amsmath,amsthm,amssymb,graphicx,enumerate,here,color}
\usepackage[all]{xy}
\allowdisplaybreaks

\theoremstyle{definition}
\newtheorem{dfn}{\bf Definition}[section]
\newtheorem{thm}[dfn]{\bf Theorem}

\newtheorem{rmk}[dfn]{\bf Remark}
\newtheorem{prop}[dfn]{\bf Propositon}
\newtheorem{exam}[dfn]{\bf Example}

\def\C{\mathbb{C}}
\def\Z{\mathbb{Z}}
\def\R{\mathbb{R}}
\def\Hom{\mathrm{Hom}}
\def\End{\mathrm{End}}
\def\Inv{\mathrm{Inv}}
\def\W{\mathrm{W}_{G_2}}

\title{Link invariant and $G_2$ web space}
\author{Takuro Sakamoto}
\author{Yasuyoshi Yonezawa}
\date{\today}

\begin{document}
\maketitle
\begin{abstract}
In this paper, we reconstruct Kuperberg's $G_2$ web space \cite{kuperberg1,kuperberg2}.
We introduce a new web (a trivalent diagram) and new relations between Kuperberg's web diagrams and the new diagram.
Using the $G_2$ webs, we define crossing formulas corresponding to $R$-matrices associated to some $G_2$ irreducible representations and calculate $G_2$ quantum link invariant for some torus links.
\end{abstract}
\section{Introduction}
Suppose that $U_q(G_2)$ is the quantum group of type $G_2$, where $q\in \C$ is neither zero nor a root of unity \cite{drinfeld,jimbo}.
An invariant theory of tensor representations of $U_q(G_2)$ fundamental representations is studied in a skein theoretic approach by Kuperberg \cite{kuperberg2} and in a representation theoretic approach by Lehrer--Zhang \cite{LZ} (The invariant theory for exceptional Lie group $G_2$ is studied by Schwarz, Huang--Zhu \cite{HZ,Sch}).
As an application of the study, we obtain Reshetikhin--Turaev's quantum link invariant ($R$-matrix) associated to $U_q(G_2)$ \cite{RT} (The $G_2$ quantum link invariant is also obtained in a planar algebra approach by Morrison--Peters--Snyder \cite{MPS}).

In the Kuperberg's approach, we introduce diagrams in Figure \ref{g2-web}, called $G_2$ web which is a diagrammatization of intertwiners between tensor representations of $U_q(G_2)$ fundamental representations \cite{kuperberg1,kuperberg2}.
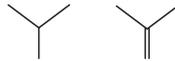
\begin{figure}[htb]
	\parbox[c]{30pt}
	{\makebox[30pt]{\txt{\scalebox{0.2}{\input{figure/web-1-11}}}}}
	\ \text{，}	
	\parbox[c]{30pt}
	{\makebox[30pt]{\txt{\scalebox{0.2}{\input{figure/web-2-11}}}}}
\caption{Kuperberg's web diagram}\label{g2-web}
\end{figure}
\\
The diagrams correspond ot intertwiners in $\Hom_{U_q(G_2)}(V_{\varpi_1},V_{\varpi_1}\otimes V_{\varpi_1})$ and $\Hom_{U_q(G_2)}(V_{\varpi_2},V_{\varpi_1}\otimes V_{\varpi_1})$, where $V_{\varpi_1}$ is the first fundamental representation and $V_{\varpi_2}$ is the second fundamental representation.
\\
\indent
The purpose of this work is to reconstruct Kuperberg's web diagram.
In Section \ref{web} we introduce a new web diagram in Figure \ref{g2-new-web} which corresponds to an intertwiner in $\Hom_{U_q(G_2)}(V_{\varpi_2},V_{\varpi_2}\otimes V_{\varpi_2})$ and show relations between Kuperberg's web diagrams and the new web diagram.
\begin{figure}[htb]
$\parbox[c]{30pt}{\makebox[30pt]{\txt{\scalebox{0.2}{\input{figure/web-2-22}}}}}$
\caption{New web}\label{g2-new-web}
\end{figure}
\\
\indent
In Section \ref{web-space} we define $G_2$ web space $\W$ which is a vector space composed of the above web diagrams and show the $G_2$ web space is isomorphic to a hom space between tensor representations of $U_{q}(G_2)$ fundamental representations.
\\
\indent
In Sections \ref{braid-action} we have crossing formulas as an expression by $G_2$ web diagrams\footnote{Remark that the first three crossing formulas are the same as Kuperberg's formulas \cite{kuperberg2} but his crossing formula of double edges contains an error.} which is related to $R$-matrices associated to $U_q(G_2)$ fundamental representations.
\begin{eqnarray*}
\label{r-11}\txt{\scalebox{0.2}{\input{figure/web-r11p}}}&=&
\frac{q^3}{[2]}\txt{\scalebox{0.2}{\input{figure/web-11-11}}}
+\frac{q^{-3}}{[2]}\txt{\scalebox{0.2}{\input{figure/web-11-0-11}}}
+\frac{q^{-1}}{[2]}\txt{\scalebox{0.2}{\input{figure/web-11-1-11-tate}}}
+\frac{q}{[2]}\txt{\scalebox{0.2}{\input{figure/web-11-1-11-yoko}}}
\\
\label{r-12}\txt{\scalebox{0.2}{\input{figure/web-r12p}}}&=&
\frac{q^{3}}{[3]}\txt{\scalebox{0.2}{\input{figure/web-21-1-12-yoko}}}
+\frac{q^{-3}}{[3]}\txt{\scalebox{0.2}{\input{figure/web-12-1-21-tate}}}
+\frac{1}{[2][3]}\txt{\scalebox{0.2}{\input{figure/web-12-1111-21}}}
\\
\label{r-21}\txt{\scalebox{0.2}{\input{figure/web-r21p}}}&=&
\frac{q^{3}}{[3]}\txt{\scalebox{0.2}{\input{figure/web-12-1-21-yoko}}}
+\frac{q^{-3}}{[3]}\txt{\scalebox{0.2}{\input{figure/web-21-1-12-tate}}}
+\frac{1}{[2][3]}\txt{\scalebox{0.2}{\input{figure/web-21-1111-12}}}
\\
\label{r-22}\txt{\scalebox{0.2}{\input{figure/web-r22p}}}&=&
\frac{(q^{10}-q^{6}-q^{4})[4][6]}{[2][12]}\txt{\scalebox{0.2}{\input{figure/web-22-22}}}
+\frac{(q^{-10}-q^{-6}-q^{-4})[4][6]}{[2][12]}\txt{\scalebox{0.2}{\input{figure/web-22-0-22}}}\\
\nonumber
&&
+\frac{q^{-3}[3][4]^2[6]^2}{[2]^2[12]^2}\txt{\scalebox{0.2}{\input{figure/web-22-2-22-tate}}}
+\frac{q^{3}[3][4]^2[6]^2}{[2]^2[12]^2}\txt{\scalebox{0.2}{\input{figure/web-22-2-22-yoko}}}
+\frac{1}{[3]}\txt{\scalebox{0.2}{\input{figure/web-22-1111-22}}}
\end{eqnarray*}

The above crossing formulas induce a braid group action on $G_2$ web space $W_{G_2}$.
Moreover, we show a relation between $G_2$ web diagrams and projectors between hom space between tensor representations and we also have crossing formulas as an expression by the projectors in Section \ref{matrixandprojector}.
The expression is useful for calculating $G_2$ quantum invariant for typical links.
In Section \ref{ex-torus}, we calculate $G_2$ quantum invariant of torus links $T(2,n)$.
\section{$G_2$ web}\label{web}
First, we introduce $G_2$ webs for defining $G_2$ web space.
\begin{dfn}[$G_2$ web]\label{relation}
Let $q\in\mathbb{C}$ be neither zero nor a root of unity.
Denote by $[n]$ for $n\in \mathbb{Z}_{\geq 0}$ the $q$-integer $\frac{q^{n}-q^{-n}}{q-q^{-1}}$ and put $[n]!:=[n][n-1]\cdots[1]$ and $\left[m\atop n\right]:=\displaystyle{\frac{[m]!}{[n]![m-n]!}}$ for $0\le n\le m$.
\\
{\it Elementary $G_{2}$ webs} are the following arc diagrams and trivalent diagrams
$$
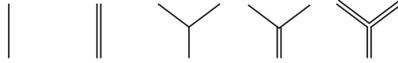
\txt{\scalebox{0.2}{\input{figure/web-1}}}\quad
\txt{\scalebox{0.2}{\input{figure/web-2}}}\quad
\txt{\scalebox{0.2}{\input{figure/web-1-11}}}\quad
\txt{\scalebox{0.2}{\input{figure/web-2-11}}}\quad
\txt{\scalebox{0.2}{\input{figure/web-2-22}}}
$$
A {\it $G_{2}$ web} is a planar diagram obtained by operations, which are gluing mutual boundaries of two single edges or two double edges of some elementary $G_2$ webs and taking union of diagrams, with the following relations:
\\
(Loop relation)
\begin{eqnarray*}
\txt{\scalebox{0.2}{\input{figure/web-0-11-0}}}
=\frac{[2][7][12]}{[4][6]}
\label{k1}
\end{eqnarray*}
(Monogon relations)
\begin{eqnarray*}
\txt{\scalebox{0.2}{\input{figure/web-1-11-0}}}
=0
\label{k3}
\qquad
\txt{\scalebox{0.2}{\input{figure/web-2-11-0}}}
=0
\qquad
\end{eqnarray*}
(Digon relations)
\begin{eqnarray*}
\txt{\scalebox{0.2}{\input{figure/web-1-11-1}}}
&=&-\frac{[3][8]}{[4]}
\txt{\scalebox{0.2}{\input{figure/web-1}}}
\qquad
\txt{\scalebox{0.2}{\input{figure/web-2-11-2}}}
=-[2][3]
\txt{\scalebox{0.2}{\input{figure/web-2}}}
\end{eqnarray*}
(Triangle relations)
\begin{eqnarray*}
\txt{\scalebox{0.2}{\input{figure/web-11-1-11-1}}}
&=&\frac{[6]}{[2]}
\txt{\scalebox{0.2}{\input{figure/web-11-1}}}
\label{k5}
\qquad
\txt{\scalebox{0.2}{\input{figure/web-22-1-11-1}}}
=0
\qquad
\txt{\scalebox{0.2}{\input{figure/web-22-1-11-2}}}
=\frac{[3]^2[4][6]}{[2][12]}
\txt{\scalebox{0.2}{\input{figure/web-22-2}}}
\end{eqnarray*}
(Double edge elimination)
\begin{equation*}
\txt{\scalebox{0.2}{\input{figure/web-11-2-11}}}
=
-
\frac{[3]}{[2]}
\txt{\scalebox{0.2}{\input{figure/web-11-11}}}
+
\frac{[3][4][6]}{[2]^2[12]}
\txt{\scalebox{0.2}{\input{figure/web-11-0-11}}}
+
\frac{1}{[2]}
\txt{\scalebox{0.2}{\input{figure/web-11-1-11-tate}}}
+
\frac{[3]}{[2]}
\txt{\scalebox{0.2}{\input{figure/web-11-1-11-yoko}}}
\label{k8}
\end{equation*}
\end{dfn}
Using the above relations, we obtain the following additional relations.
\begin{prop}\label{relation2}
(Loop relation)
\begin{eqnarray*}
\txt{\scalebox{0.2}{\input{figure/web-0-22-0}}}
=\frac{[7][8][15]}{[3][4][5]}
\label{k2}
\end{eqnarray*}
(Monogon relation)
\begin{eqnarray*}
\txt{\scalebox{0.2}{\input{figure/web-2-22-0}}}
=0
\end{eqnarray*}
(Digon relations)
\begin{eqnarray*}
\txt{\scalebox{0.2}{\input{figure/web-2-11-1}}}
=0
\qquad
\txt{\scalebox{0.2}{\input{figure/web-1-12-1}}}
=-\frac{[6][8][15]}{[5][12]}
\txt{\scalebox{0.2}{\input{figure/web-1}}}
\qquad
\txt{\scalebox{0.2}{\input{figure/web-2-22-2}}}
=-\frac{[2]^2[12][18]}{[3]^2[4][9]}
\txt{\scalebox{0.2}{\input{figure/web-2}}}
\end{eqnarray*}
(Triangle relations)
\begin{eqnarray*}
\txt{\scalebox{0.2}{\input{figure/web-11-1-11-2}}}
&=&-[3]
\txt{\scalebox{0.2}{\input{figure/web-11-2}}}
\qquad
\txt{\scalebox{0.2}{\input{figure/web-11-2-11-1}}}
=-\frac{[4][6][15]}{[5][12]}
\txt{\scalebox{0.2}{\input{figure/web-11-1}}}
\\
\txt{\scalebox{0.2}{\input{figure/web-11-2-11-2}}}
&=&-\frac{[3][4][6](q^{2}-2+q^{-2})}{[12]}
\txt{\scalebox{0.2}{\input{figure/web-11-2}}}
\qquad
\txt{\scalebox{0.2}{\input{figure/web-11-1-22-2}}}
=\frac{[6][18]}{[3][9]}
\txt{\scalebox{0.2}{\input{figure/web-11-2}}}
\\
\txt{\scalebox{0.2}{\input{figure/web-22-2-22-2}}}
&=&-\frac{[2][12](q^{8}-q^{2}+1-q^{-2}+q^{-8})}{[3][6]}
\txt{\scalebox{0.2}{\input{figure/web-22-2}}}
\end{eqnarray*}
(Square relations)
\begin{eqnarray*}
\txt{\scalebox{0.2}{\input{figure/web-11-1111-11}}}
&=&
[3]
\txt{\scalebox{0.2}{\input{figure/web-11-11}}}
+[3]
\txt{\scalebox{0.2}{\input{figure/web-11-0-11}}}
-\frac{[4]}{[2]}
\txt{\scalebox{0.2}{\input{figure/web-11-1-11-tate}}}
-\frac{[4]}{[2]}
\txt{\scalebox{0.2}{\input{figure/web-11-1-11-yoko}}}
\label{k6}
\\
\txt{\scalebox{0.2}{\input{figure/web-11-1211-11}}}
&=&
\txt{\scalebox{0.2}{\input{figure/web-11-1121-11}}}\\
&=&
\frac{[3][7]}{[2]}
\txt{\scalebox{0.2}{\input{figure/web-11-11}}}
+
\frac{[3]}{[2]}
\txt{\scalebox{0.2}{\input{figure/web-11-0-11}}}
\\
&&\quad
-
\frac{[4][6](q^{6}-q^{2}-1-q^{-2}+q^{-6})}{[2]^2[12]}
\txt{\scalebox{0.2}{\input{figure/web-11-1-11-tate}}}
+
\frac{[7]}{[2]}
\txt{\scalebox{0.2}{\input{figure/web-11-1-11-yoko}}}
\\
\txt{\scalebox{0.2}{\input{figure/web-11-1221-11}}}
&=&
\frac{[3][4][6](q^{14}+q^{8}+2q^{4}-q^{2}+1-q^{-2}+2q^{-4}+q^{-8}+q^{-14})}{[2][12]}
\txt{\scalebox{0.2}{\input{figure/web-11-11}}}
\notag
\\
&&
+
\frac{[3][4]^2[6]^2(q^{6}-2q^{4}+q^{2}+1+q^{-2}-2q^{-4}+q^{-6})}
{[2]^{2}[12]^2}
\txt{\scalebox{0.2}{\input{figure/web-11-0-11}}}
\notag
\\
&&
-
\frac{[4][6](q^{4}-2q^{2}+1-2q^{-2}+q^{-4})}
{[2][12]}
\txt{\scalebox{0.2}{\input{figure/web-11-1-11-tate}}}\notag
\\
&&
-
\frac{[4][6](q^{12}+q^{10}+q^{6}-q^{4}+q^{2}-1+q^{-2}-q^{-4}+q^{-6}+q^{-10}+q^{-12})}{[2][12]}
\txt{\scalebox{0.2}{\input{figure/web-11-1-11-yoko}}}
\\
\txt{\scalebox{0.2}{\input{figure/web-11-1111-12}}}
&=&
\txt{\scalebox{0.2}{\input{figure/web-11-1-12-tate}}}
+
\txt{\scalebox{0.2}{\input{figure/web-11-1-12-yoko}}}
\\
\txt{\scalebox{0.2}{\input{figure/web-11-1111-22}}}
&=&
\frac{[3][4][6]}{[12]}
\txt{\scalebox{0.2}{\input{figure/web-11-0-22}}}
+
\frac{[3][4][6]}{[12]}
\txt{\scalebox{0.2}{\input{figure/web-11-2-22}}}
+
\txt{\scalebox{0.2}{\input{figure/web-11-1-22}}}
\\
\txt{\scalebox{0.2}{\input{figure/web-11-1211-12}}}
&=&
\frac{[4][6]^2}{[2][3][12]}
\txt{\scalebox{0.2}{\input{figure/web-11-1-12-tate}}}
-[4]
\txt{\scalebox{0.2}{\input{figure/web-11-1-12-yoko}}}
\\
\txt{\scalebox{0.2}{\input{figure/web-11-2111-22}}}
&=&
\frac{[3][4][6][10]}{[5][12]}
\txt{\scalebox{0.2}{\input{figure/web-11-0-22}}}
+
\frac{[3][4][6]}{[2][12]}
\txt{\scalebox{0.2}{\input{figure/web-11-2-22}}}
+
\frac{[4][6]^2}{[2][3][12]}
\txt{\scalebox{0.2}{\input{figure/web-11-1-22}}}
\\
\txt{\scalebox{0.2}{\input{figure/web-11-1122-12}}}
&=&
\frac{[2][18]}{[3][9]}
\txt{\scalebox{0.2}{\input{figure/web-11-1-12-tate}}}
+
\frac{[2][18]}{[3][9]}
\txt{\scalebox{0.2}{\input{figure/web-11-1-12-yoko}}}
\\
\txt{\scalebox{0.2}{\input{figure/web-21-1122-12}}}
&=&
-\txt{\scalebox{0.2}{\input{figure/web-12-1-21-yoko}}}
-\txt{\scalebox{0.2}{\input{figure/web-21-1-12-tate}}}
-\frac{[12]}{[3][6]}\txt{\scalebox{0.2}{\input{figure/web-21-1111-12}}}
\\
\txt{\scalebox{0.2}{\input{figure/web-11-1222-22}}}
&=&
\frac{[2][18]}{[3][9]}
\txt{\scalebox{0.2}{\input{figure/web-11-0-22}}}
+\frac{(q^{10}+q^{8}-q^{2}-1-q^{-2}+q^{-8}+q^{-10})}{[3]}
\txt{\scalebox{0.2}{\input{figure/web-11-2-22}}}\notag
\\
&&\quad
-\frac{[2][5][12][18]}{[3]^2[4][6][9]}
\txt{\scalebox{0.2}{\input{figure/web-11-1-22}}}
\\
\txt{\scalebox{0.2}{\input{figure/web-22-2222-22}}}
&=&
\frac{[2]^4[5][12]^2(q^{2}-2+q^{-2})}{[3]^2[4]^2[6]^2}
\bigg(
\txt{\scalebox{0.2}{\input{figure/web-22-22}}}
+
\txt{\scalebox{0.2}{\input{figure/web-22-0-22}}}
\bigg)\notag
\\
&&
-\frac{(q^{6}-q^{4}-1-q^{-4}+q^{-6})[2]^2[12]}{[3][4][6]}
\bigg(
\txt{\scalebox{0.2}{\input{figure/web-22-2-22-tate}}}
+
\txt{\scalebox{0.2}{\input{figure/web-22-2-22-yoko}}}
\bigg)
+
\frac{[2]^3[5][12]^4}{[3]^4[4]^3[6]^4}
\txt{\scalebox{0.2}{\input{figure/web-22-1111-22}}}
\end{eqnarray*}
(Pentagon relation)
\begin{eqnarray*}
\parbox[c]{30pt}
{\makebox[30pt]{\txt{\scalebox{0.2}{\input{figure/web-pentagon}}}}}
&=&
\bigg(
\parbox[c]{30pt}
{\makebox[30pt]{\txt{\scalebox{0.2}{\input{figure/web-pentagon-1}}}}}
+
\parbox[c]{30pt}
{\makebox[30pt]{\txt{\scalebox{0.2}{\input{figure/web-pentagon-2}}}}}
+
\parbox[c]{30pt}
{\makebox[30pt]{\txt{\scalebox{0.2}{\input{figure/web-pentagon-3}}}}}
+
\parbox[c]{30pt}
{\makebox[30pt]{\txt{\scalebox{0.2}{\input{figure/web-pentagon-4}}}}}
+
\parbox[c]{30pt}
{\makebox[30pt]{\txt{\scalebox{0.2}{\input{figure/web-pentagon-5}}}}}
\bigg)\notag
\\
&&\quad
-
\bigg(
\parbox[c]{30pt}
{\makebox[30pt]{\txt{\scalebox{0.2}{\input{figure/web-pentagon-6}}}}}
+
\parbox[c]{30pt}
{\makebox[30pt]{\txt{\scalebox{0.2}{\input{figure/web-pentagon-7}}}}}
+
\parbox[c]{30pt}
{\makebox[30pt]{\txt{\scalebox{0.2}{\input{figure/web-pentagon-8}}}}}
+
\parbox[c]{30pt}
{\makebox[30pt]{\txt{\scalebox{0.2}{\input{figure/web-pentagon-9}}}}}
+
\parbox[c]{30pt}
{\makebox[30pt]{\txt{\scalebox{0.2}{\input{figure/web-pentagon-10}}}}}
\bigg)
\label{k7}
\end{eqnarray*}
\end{prop}
\begin{proof}[\rm{\bf Sketch of proof}]
Applying the relation (Double edge elimination) or its rearrangement
\begin{equation*}
\txt{\scalebox{0.2}{\input{figure/web-11-1-11-yoko}}}
=
\txt{\scalebox{0.2}{\input{figure/web-11-11}}}
-
\frac{[4][6]}{[2][12]}
\txt{\scalebox{0.2}{\input{figure/web-11-0-11}}}
-
\frac{1}{[3]}
\txt{\scalebox{0.2}{\input{figure/web-11-1-11-tate}}}
+
\frac{[2]}{[3]}
\txt{\scalebox{0.2}{\input{figure/web-11-2-11}}}
\end{equation*}
to a diagram, we obtain relations of this proposition by relations of Definition \ref{relation}.
If we can not apply the elimination or the rearrangement to a diagram, we first create single edges in the diagram by relations
$$
\txt{\scalebox{0.2}{\input{figure/web-2}}}
=-\frac{1}{[2][3]}
\txt{\scalebox{0.2}{\input{figure/web-2-11-2}}}
\qquad
\txt{\scalebox{0.2}{\input{figure/web-22-2}}}
=\frac{[2][12]}{[3]^2[4][6]}
\txt{\scalebox{0.2}{\input{figure/web-22-1-11-2}}}
$$ 
and we apply the elimination or its rearrangement.
\\
\indent
For example, to the first digon relation of this proposition, we first create single edges in the diagram:
$$
\txt{\scalebox{0.2}{\input{figure/web-2-11-1}}}=-\frac{1}{[2][3]}\txt{\scalebox{0.2}{\input{figure/web-2-11-2-11-1}}}
$$
Applying (Double edge elimination) to the obtained diagram and using monogon, digon and triangle relations of Definition \ref{relation}, we obtain the first digon relation:
\begin{eqnarray*}
-\frac{1}{[2][3]}\txt{\scalebox{0.2}{\input{figure/web-2-11-2-11-1}}}
&=&-\frac{1}{[2][3]}\left(
-\frac{[3]}{[2]}\txt{\scalebox{0.2}{\input{figure/web-2-11-11-1}}}
+\frac{[3][4][6]}{[2]^2[12]}\txt{\scalebox{0.2}{\input{figure/web-2-11-0-11-1}}}
+\frac{1}{[2]}\txt{\scalebox{0.2}{\input{figure/web-2-11-1-11-1}}}
+\frac{[3]}{[2]}\txt{\scalebox{0.2}{\input{figure/web-2-11-1-11-1-yoko}}}
\right)\\
&=&-\frac{1}{[2][3]}\left(
-\frac{[3]}{[2]}
+0
+\frac{1}{[2]}\left(-\frac{[3][8]}{[4]}\right)
+\frac{[3]}{[2]}\frac{[6]}{[2]}
\right)\txt{\scalebox{0.2}{\input{figure/web-2-11-1}}}
=0.
\end{eqnarray*}
\end{proof}
\section{Web space $\W$ and invariant space of representation}\label{web-space}
In this section, we define $G_2$ web space $\W$ which is a $\C$ vector space spanned by $G_2$ webs embedded on a unit disk.
\\
\indent
Let $D$ be a closed unit disk in $\R^2$ with a fixed base point $\ast$ on the boundary $\partial D$ and $P$ be a $G_2$ web.
A {\it $G_2$ web diagram} is the image of an embedding on $D$ of a $G_2$ web $P$ such that boundaries of $P$ on $\partial D\backslash\{\ast\}$.
\\
\indent
For a given $G_2$ web diagram $W$, put the number $1$ at intersection of single edges of $W$ and $\partial D$ and put the number $2$ at intersection on double edges of $W$ and $\partial D$.
A {\it coloring} of $W$ is defined by a sequence obtained by reading numbers $1$ and $2$ on $\partial D$ clockwise from the base point $\ast$.
If $W$ has no boundary point, a coloring of $W$ is defined by the empty $\emptyset$.
Denote by $s(W)$ the coloring of $W$.
Two $G_2$ web diagram $W_1$ and $W_2$ are {\it isotopic} if there exist a base point-preserving isotopy of $D$ which moves $W_1$ to $W_2$.
\\
\indent
Example of $G_2$ web diagrams in Figure \ref{example}, we find colorings $s(W_1)=(1,1,1,1)$, $s(W_1)=(2,1,1)$, $s(W_1)=(1,1,2,1)$, $s(W_1)=(1,2,2,1,1)$.
\begin{figure}[htb]
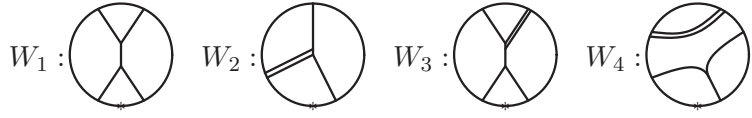

$W_1:\,$\scalebox{0.3}{\input{figure/webd-11-1-11-tate}}\quad
$W_2:\,$\scalebox{0.3}{\input{figure/webd-21-1}}\quad
$W_3:\,$\scalebox{0.3}{\input{figure/webd-11-1-12-tate}}\quad
$W_4:\,$\scalebox{0.3}{\input{figure/webd-11221}}
\\[1em]
\caption{$G_2$ web diagrams}\label{example}
\end{figure}
\\
\indent
Hereafter fix a base point as $G_2$ web diagrams in Figure \ref{example} and omit the boundary of diagrams.
\\
\indent
Write
$$
S:=\{s=(s_1,s_2,...,s_n)\,|\,n\geq 1, s_i\in \{1,2\}\, (i=1,2,...,n)\}\cup\{\emptyset\}.
$$
We define $G_2$ web space $\W(s)$ for $s\in S$ by a $\C$-linear space spanned by isotopy classes of $G_2$ web diagrams with the coloring $s$.\\
\begin{rmk}
The collection of web spaces $\{\W(s)\}_{s\in S}$ has the spader structure in the sense of Kuperberg \cite[Section 3]{kuperberg2}:
\\
(Join)
\begin{equation*}
\mu_{s,t}:\W(s)\times\W(t)\to\W(st)
\end{equation*}
(Rotation)
\begin{equation*}
\rho_{s,t}:\W(st)\to\W(ts)
\end{equation*}
(Stitch)
\begin{equation*}
\sigma_{sst}:\W(sst)\to\W(t).
\end{equation*}
\end{rmk}

\indent
For $s=(s_1,s_2,...,s_n) \in S$, let $V_s$ be a tensor representation of $G_2$ quantum group $V_{\varpi_{s_1}}\otimes V_{\varpi_{s_2}}\otimes\cdots V_{\varpi_{s_n}}$, where $V_{\varpi_{s_i}}$ is the $s_i$-th fundamental representation $(i=1,...,n)$.
\\
\indent
Following is a theorem due to \cite[Theorem 6.10]{kuperberg2}
\begin{thm}[\cite{kuperberg2}]
The vector space $\W(s)$ and $\Inv(V_s)$ have the same dimension.
\end{thm}
\begin{proof}
Replacing numbers $2$ in the coloring $s$ into $[1,1]$, we obtain a clasp sequence $C$ (See \cite{kuperberg2}). Since the web space $\W(s)$ and the clasp web space $\W(C)$ is the same dimension, we have the theorem.
\end{proof}

We denote by $\mathrm{B}(s)$ a basis of the vector space $\W(s)$, called {\it $G_2$ web basis}.
\begin{exam}
For $s=(1,1,1,1)$, $(1,2,1,2)$ and $(2,2,2,2)$, we have a $G_2$ web basis $\mathrm{B}(s)$.
\begin{eqnarray*}
\mathrm{B}(1,1,1,1)&=&
\left\{\,
\txt{\scalebox{0.2}{\input{figure/web-11-11}}},\,
\txt{\scalebox{0.2}{\input{figure/web-11-0-11}}},\,
\txt{\scalebox{0.2}{\input{figure/web-11-1-11-tate}}},\,
\txt{\scalebox{0.2}{\input{figure/web-11-1-11-yoko}}}\,
\right\}
\\
\mathrm{B}(1,2,1,2)&=&
\left\{\,
\txt{\scalebox{0.2}{\input{figure/web-21-1-12-yoko}}},\,
\txt{\scalebox{0.2}{\input{figure/web-12-1-21-tate}}},\,
\txt{\scalebox{0.2}{\input{figure/web-12-1111-21}}}\,
\right\}
\\
\mathrm{B}(2,2,2,2)&=&
\left\{
\txt{\scalebox{0.2}{\input{figure/web-22-22}}},\,
\txt{\scalebox{0.2}{\input{figure/web-22-0-22}}},\,
\txt{\scalebox{0.2}{\input{figure/web-22-2-22-tate}}},\,
\txt{\scalebox{0.2}{\input{figure/web-22-2-22-yoko}}},\,
\txt{\scalebox{0.2}{\input{figure/web-22-1111-22}}}\,
\right\}
\end{eqnarray*}
\end{exam}
\section{Braid action on $G_2$ web space $\W$}\label{braid-action}
Let $\#(s)$ be a length of a sequence $s\in S$ and define 
$$
S[n]:=\{s\in S| \#(s)=n\}.
$$
We define an action of the braid group 
$$
B_n=\left\langle b_i \,(1\leq i\leq n-1)\left| \begin{array}{ll}b_ib_j=b_jb_i &(|i-j|>1),\\ b_ib_{i+1}b_i=b_{i+1}b_ib_{i+1} &(1\leq i\leq n-2)\end{array}\right.\right\rangle
$$
on the web space
$$
\W[n]:=\bigoplus_{s\in S[n]}\W(s).
$$
\begin{dfn}\label{crossing-formula}
Four types of crossings have the following descriptions in $G_2$ web diagram:
\begin{eqnarray}
\label{r-11}\txt{\scalebox{0.2}{\input{figure/web-r11p}}}&=&
\frac{q^3}{[2]}\txt{\scalebox{0.2}{\input{figure/web-11-11}}}
+\frac{q^{-3}}{[2]}\txt{\scalebox{0.2}{\input{figure/web-11-0-11}}}
+\frac{q^{-1}}{[2]}\txt{\scalebox{0.2}{\input{figure/web-11-1-11-tate}}}
+\frac{q}{[2]}\txt{\scalebox{0.2}{\input{figure/web-11-1-11-yoko}}}
\\
\label{r-12}\txt{\scalebox{0.2}{\input{figure/web-r12p}}}&=&
\frac{q^{3}}{[3]}\txt{\scalebox{0.2}{\input{figure/web-21-1-12-yoko}}}
+\frac{q^{-3}}{[3]}\txt{\scalebox{0.2}{\input{figure/web-12-1-21-tate}}}
+\frac{1}{[2][3]}\txt{\scalebox{0.2}{\input{figure/web-12-1111-21}}}
\\
\label{r-21}\txt{\scalebox{0.2}{\input{figure/web-r21p}}}&=&
\frac{q^{3}}{[3]}\txt{\scalebox{0.2}{\input{figure/web-12-1-21-yoko}}}
+\frac{q^{-3}}{[3]}\txt{\scalebox{0.2}{\input{figure/web-21-1-12-tate}}}
+\frac{1}{[2][3]}\txt{\scalebox{0.2}{\input{figure/web-21-1111-12}}}
\\
\label{r-22}\txt{\scalebox{0.2}{\input{figure/web-r22p}}}&=&
\frac{(q^{10}-q^{6}-q^{4})[4][6]}{[2][12]}\txt{\scalebox{0.2}{\input{figure/web-22-22}}}
+\frac{(q^{-10}-q^{-6}-q^{-4})[4][6]}{[2][12]}\txt{\scalebox{0.2}{\input{figure/web-22-0-22}}}\\
\nonumber
&&
+\frac{q^{-3}[3][4]^2[6]^2}{[2]^2[12]^2}\txt{\scalebox{0.2}{\input{figure/web-22-2-22-tate}}}
+\frac{q^{3}[3][4]^2[6]^2}{[2]^2[12]^2}\txt{\scalebox{0.2}{\input{figure/web-22-2-22-yoko}}}
+\frac{1}{[3]}\txt{\scalebox{0.2}{\input{figure/web-22-1111-22}}}
\end{eqnarray}
\end{dfn}
We have the following theorem by a direct calculation.
\begin{thm}
Four types of crossing are regular isotopic (i.e. invariant under Reidemeister move 2 and 3. See \cite{kauff}).
\end{thm}
The above crossing formulas (\ref{r-11}), (\ref{r-12}), (\ref{r-21}) and (\ref{r-22}) correspond to $R$-matrices $R_{11}\in\End_{U_q(G_2)}(V_{\varpi_1}^{\otimes 2})$, $R_{12}\in\Hom_{U_q(G_2)}(V_{\varpi_1}\otimes V_{\varpi_2},V_{\varpi_2}\otimes V_{\varpi_1})$, $R_{21}\in\Hom_{U_q(G_2)}(V_{\varpi_2}\otimes V_{\varpi_1},V_{\varpi_1}\otimes V_{\varpi_2})$ and $R_{22}\in\End_{U_q(G_2)}(V_{\varpi_2}^{\otimes 2})$

Using the crossing formulas of Definition \ref{crossing-formula}, we define an action $b_i\in B_n$ on $\W[n]$ as follows: The braid group action on a direct summand $\W(s)\subset \W[n]$ is
$$
b_i:\W(s)\to \W(\sigma_i(s)),
$$
where $\sigma_i$ is the transpose between an $i$-th entry and an $i+1$-entry.
If $(s_i,s_{i+1})=(1,1)$ (resp. $(s_i,s_{i+1})=(1,2)$, $(s_i,s_{i+1})=(2,1)$, $(s_i,s_{i+1})=(2,2)$), gluing the boundaries of a $G_2$ web diagram $W\in\W(s)$ and the crossing in formula (\ref{r-11}) of Definition \ref{crossing-formula} (resp. the crossing in (\ref{r-12}), (\ref{r-21}), (\ref{r-22})) at $(s_i,s_{i+1})$ as the boundary at $s_i$ connects to the over arc of the crossing and replacing the crossing into web diagrams by the formula (\ref{r-11}) of Definition \ref{crossing-formula} (resp. formulas (\ref{r-12}), (\ref{r-21}), (\ref{r-22})), we obtain a linear sum of $G_2$ web diagrams in $\W(\sigma_i(s))$.
An action of $b_i^{-1}\in B_n$,
$$
b_i^{-1}:\W(s)\to \W(\sigma_i(s)),
$$
is defined by gluing the boundaries at $(s_i,s_{i+1})$ and the crossing as the boundary at $s_i$ connects to the under arc of the crossing and replacing the crossing into the linear sum of $G_2$ web diagrams.
\\
\indent
For example, we have the $B_5$ action on the $G_2$ web space $\W(1,2,2,1,1)$.
To the $G_2$ web diagram $W_4$ in Figure \ref{example}, the action of $b_1, b_4, b_4^{-1}\in B_5$ is
\begin{eqnarray*}
b_1(\scalebox{0.2}{\input{figure/webd-11221}})&=&\scalebox{0.2}{\input{figure/braid-action1}}
=\frac{q^3}{[3]}\scalebox{0.2}{\input{figure/braid-action1-res2}}
+\frac{q^{-3}}{[3]}\scalebox{0.2}{\input{figure/braid-action1-res1}}
+\frac{1}{[2][3]}\scalebox{0.2}{\input{figure/braid-action1-res3}}
\\[0.2em]
b_4(\scalebox{0.2}{\input{figure/webd-11221}})&=&\scalebox{0.2}{\input{figure/braid-action2}}
=\frac{q^3}{[2]}\scalebox{0.2}{\input{figure/braid-action2-res1}}
+\frac{q^{-3}}{[2]}\scalebox{0.2}{\input{figure/braid-action2-res2}}
+\frac{q^{-1}}{[2]}\scalebox{0.2}{\input{figure/braid-action2-res3}}
+\frac{q}{[2]}\scalebox{0.2}{\input{figure/braid-action2-res4}}\\
&=&\left(\frac{q^3}{[2]}-\frac{q^{-1}[3][8]}{[2][4]}+\frac{q[6]}{[2]^2}\right)\scalebox{0.2}{\input{figure/webd-11221}}
=-q^{-6}\scalebox{0.2}{\input{figure/webd-11221}}
\\[0.2em]
b_4^{-1}(\scalebox{0.2}{\input{figure/webd-11221}})&=&\scalebox{0.2}{\input{figure/braid-action2-inv}}
=\frac{q^3}{[2]}\scalebox{0.2}{\input{figure/braid-action2-res2}}
+\frac{q^{-3}}{[2]}\scalebox{0.2}{\input{figure/braid-action2-res1}}
+\frac{q^{-1}}{[2]}\scalebox{0.2}{\input{figure/braid-action2-res4}}
+\frac{q}{[2]}\scalebox{0.2}{\input{figure/braid-action2-res3}}\\
&=&\left(\frac{q^{-3}}{[2]}-\frac{q[3][8]}{[2][4]}+\frac{q^{-1}[6]}{[2]^2}\right)\scalebox{0.2}{\input{figure/webd-11221}}
=-q^{6}\scalebox{0.2}{\input{figure/webd-11221}}
\end{eqnarray*}
\section{Relation to projectors and $R$-matrix of other irreducible representations}\label{matrixandprojector}
In this section, we show a relationship between $G_2$ web diagrams and projectors in hom set $\Hom_{U_q(G_2)}(V_{\varpi}\otimes V_{\varpi'})$, where $\varpi$ and $\varpi$ are fundamental weights.
Using projectors, we construct the crossing formulas associated to other irreducible representations.
\\
\indent
Let $P_{11}[\varpi]$ be a projector in $\End_{U_q(G_2)}(V_{\varpi_1}^{\otimes 2})$ which factors through the irreducible representation with highest weight $\varpi$ and let $R_{11}$ be the $R$-matrix in $\End_{U_q(G_2)}(V_{\varpi_1}^{\otimes 2})$.
Remark that the projectors have idempotency
$$
P_{11}[\varpi]
P_{11}[\varpi']
=
\delta_{\varpi,\varpi'}P_{11}[\varpi].
$$
The description of $R_{11}$ by the projectors (see \cite{LZ}) is
$$
R_{11}=q^2 P_{11}[2\varpi_1]-q^{-6}P_{11}[\varpi_1]-P_{11}[\varpi_2]+q^{-12}P_{11}[0].
$$
\\
\indent
A relation between $G_2$ web diagrams and these projectors is 
\begin{eqnarray*}
P_{11}[2\varpi_1]&=&\txt{\scalebox{0.2}{\input{figure/web-11-11}}}
+\frac{[4]}{[3][8]}\txt{\scalebox{0.2}{\input{figure/web-11-1-11-tate}}}
+\frac{1}{[2][3]}\txt{\scalebox{0.2}{\input{figure/web-11-2-11}}}
-\frac{[4][6]}{[2][7][12]}\txt{\scalebox{0.2}{\input{figure/web-11-0-11}}}
\\
P_{11}[\varpi_1]&=&-\frac{[4]}{[3][8]}\txt{\scalebox{0.2}{\input{figure/web-11-1-11-tate}}}\\
P_{11}[\varpi_2]&=&-\frac{1}{[2][3]}\txt{\scalebox{0.2}{\input{figure/web-11-2-11}}}\\
P_{11}[0]&=&\frac{[4][6]}{[2][7][12]}\txt{\scalebox{0.2}{\input{figure/web-11-0-11}}}
\end{eqnarray*}
\indent
In other $R$-matrices of $R_{12}\in\Hom_{U_q(G_2)}(V_{\varpi_1}\otimes V_{\varpi_2},V_{\varpi_2}\otimes V_{\varpi_1})$, $R_{21}\in\Hom_{U_q(G_2)}(V_{\varpi_2}\otimes V_{\varpi_1},V_{\varpi_1}\otimes V_{\varpi_2})$ and $R_{22}\in\End_{U_q(G_2)}(V_{\varpi_2}^{\otimes 2})$, these descriptions by projectors are
\begin{eqnarray*}
R_{12}&=&q^3 P_{12}[\varpi_1+\varpi_2]+q^{-4}P_{12}[2\varpi_1]-q^{-12}P_{12}[\varpi_1]\\
R_{21}&=&q^3 P_{21}[\varpi_1+\varpi_2]+q^{-4}P_{21}[2\varpi_1]-q^{-12}P_{21}[\varpi_1]\\
R_{22}&=&q^6 P_{22}[2\varpi_2]-P_{22}[3\varpi_1]+q^{-10}P_{22}[2\varpi_1]-q^{-12}P_{22}[\varpi_2]+q^{-24}P_{22}[0],
\end{eqnarray*}
where $P_{ij}[\varpi]$, $i,j\in\{1,2\}$, is a projector in $\Hom_{U_q(G_2)}(V_{\varpi_i}\otimes V_{\varpi_j},V_{\varpi_j}\otimes V_{\varpi_i})$ which factors through the representation $V_\varpi$.
Remark that the projectors have a structure
$$
P_{12}[\varpi]
P_{21}[\varpi']
P_{12}[\varpi]
=
\delta_{\varpi,\varpi'}P_{12}[\varpi],
\quad
P_{22}[\varpi]
P_{22}[\varpi']
=
\delta_{\varpi,\varpi'}P_{22}[\varpi].
$$
\indent
A relation between $G_2$ web diagrams and projectors $P_{ij}[\varpi]$ is
\begin{eqnarray*}
P_{12}[\varpi_1+\varpi_2]&=&
\frac{1}{[3]}\txt{\scalebox{0.2}{\input{figure/web-21-1-12-yoko}}}
+\frac{[5](q^8+q^2-1+q^{-2}+q^{-8})}{[7][15]}\txt{\scalebox{0.2}{\input{figure/web-12-1-21-tate}}}
+\frac{[4]}{[2][3][7]}\txt{\scalebox{0.2}{\input{figure/web-12-1111-21}}}\\
P_{12}[2\varpi_1]&=&
\frac{1}{[2][7]}\txt{\scalebox{0.2}{\input{figure/web-12-1111-21}}}
+\frac{[3][4]}{[2][7][8]}\txt{\scalebox{0.2}{\input{figure/web-12-1-21-tate}}}\\
P_{12}[\varpi_1]&=&-\frac{[5][12]}{[6][8][15]}\txt{\scalebox{0.2}{\input{figure/web-12-1-21-tate}}}\\
P_{22}[2\varpi_2]&=&
\frac{[3][4][5](q^2-2+q^{-2})}{[12]}\txt{\scalebox{0.2}{\input{figure/web-22-22}}}
-\frac{[3]^2[4][5][14]}{[7][8][12][15]}\txt{\scalebox{0.2}{\input{figure/web-22-0-22}}}\\
&&+\frac{[3]^2[4]^2[6][9]([4][14]-{7})}{[2]^2[7][8][12]^2[18]}\txt{\scalebox{0.2}{\input{figure/web-22-2-22-tate}}}
+\frac{[3]^2[4]^2[6]}{[2]^2[12]^2}\txt{\scalebox{0.2}{\input{figure/web-22-2-22-yoko}}}\\
&&+\frac{[5]}{[6][8]}\txt{\scalebox{0.2}{\input{figure/web-22-1111-22}}}\\
P_{22}[3\varpi_1]&=&
\frac{[3][4]}{[12]}\txt{\scalebox{0.2}{\input{figure/web-22-22}}}
+\frac{[2][3][4]^2[5]}{[8][10][12]}\txt{\scalebox{0.2}{\input{figure/web-22-0-22}}}
-\frac{[3]^4[4]^2[5]}{[2]^2[10][12]^2}\txt{\scalebox{0.2}{\input{figure/web-22-2-22-tate}}}\\
&&-\frac{[3]^2[4]^2[6]}{[2]^2[12]^2}\txt{\scalebox{0.2}{\input{figure/web-22-2-22-yoko}}}
-\frac{[4][5]}{[2][6][10]}\txt{\scalebox{0.2}{\input{figure/web-22-1111-22}}}\\
P_{22}[2\varpi_1]&=&
-\frac{[2][3]^2[4][5][6]}{[7][8][10][12]}\txt{\scalebox{0.2}{\input{figure/web-22-0-22}}}
+\frac{[3]^3[4]^2[5][6]^2}{[2]^3[8][10][12]^2}\txt{\scalebox{0.2}{\input{figure/web-22-2-22-tate}}}
+\frac{[5]}{[8][10]}\txt{\scalebox{0.2}{\input{figure/web-22-1111-22}}}\\
P_{22}[\varpi_2]&=&-\frac{[3]^2[4][9]}{[2]^2[12][18]}\txt{\scalebox{0.2}{\input{figure/web-22-2-22-tate}}}\\
P_{22}[0]&=&\frac{[3][4][5]}{[7][8][15]}\txt{\scalebox{0.2}{\input{figure/web-22-0-22}}}\\
\end{eqnarray*}
A projector $P_{21}[\varpi]$ is symmetry of the description of $P_{12}[\varpi]$ by $G_2$ web diagrams.
In other words,
$$
P_{21}[\varpi]=R_{21}P_{12}[\varpi]R_{12}^{-1}.
$$
\indent
Using the description of projectors, we obtain a crossing formula associated to the irreducible representation $V_\varpi$.
For example, using the description of $P_{11}[2\varpi_1]$ by $G_2$ web diagrams, we obtain a crossing formula with coloring $2\varpi_1$
\begin{eqnarray*}
\txt{\scalebox{0.2}{\input{figure/web-r2times1-1p}}}&=&
\txt{\scalebox{0.2}{\input{figure/web-r2times1-1-fusion1}}}
-\frac{[4][6]}{[2][7][12]}\txt{\scalebox{0.2}{\input{figure/web-r2times1-1-fusion2}}}
+\frac{[4]}{[3][8]}\txt{\scalebox{0.2}{\input{figure/web-r2times1-1-fusion3}}}
+\frac{1}{[2][3]}\txt{\scalebox{0.2}{\input{figure/web-r2times1-1-fusion4}}}
\\
\txt{\scalebox{0.2}{\input{figure/web-r2times1-2p}}}&=&
\txt{\scalebox{0.2}{\input{figure/web-r2times1-2-fusion1}}}
-\frac{[4][6]}{[2][7][12]}\txt{\scalebox{0.2}{\input{figure/web-r2times1-2-fusion2}}}
+\frac{[4]}{[3][8]}\txt{\scalebox{0.2}{\input{figure/web-r2times1-2-fusion3}}}
+\frac{1}{[2][3]}\txt{\scalebox{0.2}{\input{figure/web-r2times1-2-fusion4}}}
\end{eqnarray*}
We also have a formula for the following crossing
$$
\txt{\scalebox{0.2}{\input{figure/web-r2times1-2times1p}}}
$$
as a linear sum of $16$ diagrams.
Similarly, we have crossing formulas with colorings $\varpi_1+\varpi_2$, $2\varpi_2$ and $3\varpi_1$ using the projectors $P_{12}[\varpi_1+\varpi_2]$, $P_{22}[2\varpi_2]$ and $P_{22}[3\varpi_1]$.
\\
\indent
An open problem is to construct projectors which factor through other irreducible representations as a linear sum of $G_2$ web diagrams.

\section{$G_2$ quantum invariant of torus links}\label{ex-torus}
\indent
We have the evaluations for positive and negative crossings curl (diagrams in Reidemeister move 1) by crossing formulas (\ref{r-11}) and (\ref{r-22}) of Definition \ref{crossing-formula}
\begin{eqnarray*}
&&\txt{\input{figure/p-crossing}}=q^{12}\txt{\input{figure/line}}\quad
\txt{\input{figure/n-crossing}}=q^{-12}\txt{\input{figure/line}}
\\
&&\txt{\input{figure/p-d-crossing}}=q^{24}\txt{\input{figure/d-line}}\quad
\txt{\input{figure/n-d-crossing}}=q^{-24}\txt{\input{figure/d-line}}
\end{eqnarray*}
Therefore, to obtain $G_2$ quantum invariant of an oriented link, we need to normalize crossing formulas of Definition \ref{crossing-formula}.\\
\indent
Let $L$ be an oriented link with $k$ components $(L_1,L_2,...,L_k)$, let $D$ be a link diagram of $L$ (forgetting the orientation) and $(D_1,D_2,...,D_k)$ be the image of $(L_1,L_2,...,L_k)$ by the projection $L\to D$.\\
\indent
Using crossing formulas, we define a polynomial evaluation for a link diagram $D$, denoted by $\langle D \rangle_{(\varpi_{i_1},\varpi_{i_2},...,\varpi_{i_k})}$, $i_j\in \{1,2\}$ and $j=1,...,k$, as follows: Replacing a diagram component $D_j$ into the double line of $G_2$ web diagram if $\varpi_{i_j}=\varpi_2$ (we regard a diagram component $D_j$ as the single line of $G_2$ web diagram if $\varpi_{i_j}=\varpi_1$), applying the crossing formulas of Definition \ref{crossing-formula} to all crossings of the replaced diagram of $D$ and evaluating a linear sum of $G_2$ web diagrams by relations of Definition \ref{relation} and Proposition \ref{relation2}, we obtain a polynomial.
\begin{thm}
For an oriented link L,
$$
(q^{-12})^{\omega_{11}(D)}
(q^{-24})^{\omega_{22}(D)}
\langle D \rangle_{(\varpi_{i_1},\varpi_{i_2},...,\varpi_{i_k})}
$$
is a link invariant of $L$, where $D$ is a link diagram of $L$ and $\omega_{11}(D)$ (resp. $\omega_{22}(D)$) is the number of positive crossings of single edge in $D$ minus the number of negative crossings of single edge (resp. the number of positive crossings of double edge minus the number of negative crossings of double edge).
\end{thm}
The link invariant is called $G_2$ quantum invariant associated to the $G_2$ fundamental representations, $G_2$ quantum invariant for short.
Denote by $P_{(\varpi_{i_1},\varpi_{i_2},...,\varpi_{i_k})}(L)$ the $G_2$ quantum invariant of an oriented link $L$.

\indent
We show $G_2$ quantum invariant of a torus link $T(2,n)$ in Figure \ref{torus-pretzel}.
\begin{figure}[htb]
$\txt{\input{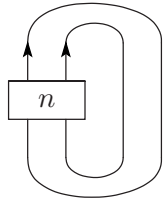}}$
\caption{Torus link $T(2,n)$
}\label{torus-pretzel}
\end{figure}
\\
\indent
Let $Cr(n)$, where $n\in\Z$, be a tangle diagram with $n$-crossing in Figure \ref{crossing}.
The evaluation $\langle Cr(n)\rangle_{(\varpi_{1},\varpi_{1})}$ is
\begin{eqnarray*}
\langle Cr(n)\rangle_{(\varpi_{1},\varpi_{1})}
&=&(q^2 P_{11}[2\varpi_1]-q^{-6}P_{11}[\varpi_1]-P_{11}[\varpi_2]+q^{-12}P_{11}[0])^n\\
&=&q^{2n}\txt{\scalebox{0.2}{\input{figure/web-11-11}}}
+\frac{[4][6]}{[2][7][12]}(-q^{2n}+q^{-12n})\txt{\scalebox{0.2}{\input{figure/web-11-0-11}}}\\
&&+\frac{[4]}{[3][8]}(q^{2n}-(-q^{-6})^n)\txt{\scalebox{0.2}{\input{figure/web-11-1-11-tate}}}
+\frac{1}{[2][3]}(q^{2n}-(-1)^n)\txt{\scalebox{0.2}{\input{figure/web-11-2-11}}}.
\end{eqnarray*} 

\begin{figure}[htb]
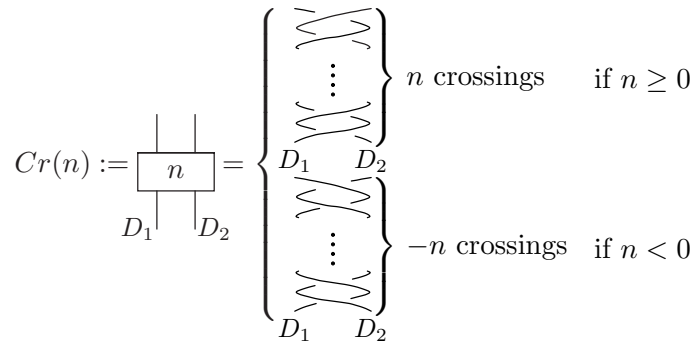

$
Cr(n):=\txt{\input{figure/n-crossing1}}=
\left\{\begin{array}{ll}
\txt{$\left.\txt{\input{figure/n-crossing2}}\right\}$ $n$ crossings} &\txt{if $n\geq 0$}
\\[3em]
\txt{$\left.\txt{\input{figure/n-crossing3}}\right\}$ $-n$ crossings} &\txt{if $n< 0$}
\end{array}
\right.
$
\caption{$n$-crossing}\label{crossing}
\end{figure}

Using the evaluation $\langle Cr(n)\rangle_{(\varpi_1,\varpi_1)}$, the $G_2$ invariant $P_{(\varpi_1,\varpi_1)}(T(2,n))$, where $n\in \Z$,
is
\begin{eqnarray*}
&&q^{-12n}\langle T(2,n)\rangle_{(\varpi_1,\varpi_1)}\\
&=&q^{-10n}\frac{[3][6][12][15]}{[4][5][6]}+(-q^{-12})^n\frac{[7][8][15]}{[3][4][5]}+(-q^{-18})^n\frac{[2][7][12]}{[4][6]}+q^{-24n}.
\end{eqnarray*}

Similarly, calculating $\langle Cr(n)\rangle_{(\varpi_1,\varpi_2)}$ for an even number $n$ and $\langle Cr(n)\rangle_{(\varpi_2,\varpi_2)}$ for $n\in\Z$ (the details are left to the reader), $G_2$ invariant $P_{(\varpi_1,\varpi_2)}(T(2,n))$, where $n$ is an even integer, is
\begin{eqnarray*}
\langle T(2,n)\rangle_{(\varpi_1,\varpi_2)}
&=&q^{3n}\frac{[2][8][10][12][18]}{[3][4][5][9]}+q^{-4n} \frac{[3][12][15]}{[4][5]}+(-q^{-12})^n\frac{[2][7][12]}{[4][6]}.
\end{eqnarray*}
and $G_2$ invariant $P_{(\varpi_2,\varpi_2)}(T(2,n))$, where $n\in\Z$, is
\begin{eqnarray*}
q^{-24n}\langle T(2,n)\rangle_{(\varpi_2,\varpi_2)}
&=&q^{-18 n}\frac{[10][11][12][21]}{[3][4][5][6]}+(-q^{-24})^n\frac{[7][11][15][18]}{[5][6][9]}
\\
&&+q^{-34n}\frac{[3][12][15]}{[4][5]} + (-q^{-36})^n \frac{[7][8][15]}{[3][4][5]}+ q^{-48 n}.
\end{eqnarray*}
\\
Acknowledgements:
We thank Greg Kuperberg for comments on the earlier version of our paper.
We also thank Takahiro Hayashi, Tomoki Nakanishi, Soichi Okada and Kenichi Shimizu for comments on our work.
The second author thanks Yoshiyuki Kimura and Jun Murakami for helpful discussions.

\end{document}

%% file: figure/web-r11p.tex
\unitlength 0.1in
\begin{picture}( 13.0000, 15.0000)(  2.0000,-18.0000)
\special{pn 40}%
\special{pa 100 100}%
\special{pa 650 650}%
\special{fp}%
\special{pa 1050 1050}%
\special{pa 1600 1600}%
\special{fp}%
\special{pa 100 1600}%
\special{pa 1600 100}%
\special{fp}%
\end{picture}%

%% file: figure/web-11-11.tex
{\unitlength 0.1in%
\begin{picture}(16.0000,16.0000)(2.0000,-20.0000)%
\special{pn 40}%
\special{ar 2400 1000 1000 1000 2.2142974 4.0688879}%
\special{pn 40}%
\special{ar -400 1000 1000 1000 5.3558901 0.9272952}%
\end{picture}}%

%% file: figure/web-11-0-11.tex
{\unitlength 0.1in%
\begin{picture}(16.0000,16.0000)(2.0000,-20.0000)%
\special{pn 40}%
\special{ar 1000 2400 1000 1000 3.7850938 5.6396842}%
\special{pn 40}%
\special{ar 1000 -400 1000 1000 0.6435011 2.4980915}%
\end{picture}}%

%% file: figure/web-r21p.tex
\unitlength 0.1in
\begin{picture}( 13.0000, 15.0000)(  2.0000,-18.0000)
\special{pn 40}%
\special{pa 100 100}%
\special{pa 650 650}%
\special{fp}%
\special{pa 1050 1050}%
\special{pa 1600 1600}%
\special{fp}%
\special{pa 150 1650}%
\special{pa 1650 150}%
\special{fp}%
\special{pa 50 1550}%
\special{pa 1550 50}%
\special{fp}%
\end{picture}%

%% file: figure/web-22-22.tex
{\unitlength 0.1in%
\begin{picture}(16.0000,16.0000)(2.0000,-20.0000)%
\special{pn 40}%
\special{ar 2400 1000 1000 1000 2.2142974 4.0688879}%
\special{pn 40}%
\special{ar -400 1000 1000 1000 5.3558901 0.9272952}%
\special{pn 40}%
\special{ar 2400 1000 900 900 2.2142974 4.0688879}%
\special{pn 40}%
\special{ar -400 1000 900 900 5.3558901 0.9272952}%
\end{picture}}%

%% file: figure/web-22-0-22.tex
{\unitlength 0.1in%
\begin{picture}(16.0000,16.0000)(2.0000,-20.0000)%
\special{pn 40}%
\special{ar 1000 2400 1000 1000 3.7850938 5.6396842}%
\special{pn 40}%
\special{ar 1000 -400 1000 1000 0.6435011 2.4980915}%
\special{pn 40}%
\special{ar 1000 2400 900 900 3.7850938 5.6396842}%
\special{pn 40}%
\special{ar 1000 -400 900 900 0.6435011 2.4980915}%
\end{picture}}%

%% file: figure/web-1.tex
{\unitlength 0.1in%
\begin{picture}(16.0000,16.0000)(2.0000,-18.0000)%
\special{pn 40}%
\special{pa 1000 100}%
\special{pa 1000 1500}%
\special{fp}%
\end{picture}}%

%% file: figure/web-2.tex
{\unitlength 0.1in%
\begin{picture}(16.0000,16.0000)(2.0000,-18.0000)%
\special{pn 40}%
\special{pa 1050 100}%
\special{pa 1050 1500}%
\special{fp}%
\special{pa 950 100}%
\special{pa 950 1500}%
\special{fp}%
\end{picture}}%

%% file: figure/web-1-11.tex
{\unitlength 0.1in%
\begin{picture}(16.0000,16.0000)(2.0000,-18.0000)%
\special{pn 40}%
\special{pa 200 100}%
\special{pa 1000 700}%
\special{fp}%
\special{pa 1000 700}%
\special{pa 1800 100}%
\special{fp}%
\special{pa 1000 700}%
\special{pa 1000 1500}%
\special{fp}%
\end{picture}}%

%% file: figure/web-2-11.tex
{\unitlength 0.1in%
\begin{picture}(16.0000,16.0000)(2.0000,-18.0000)%
\special{pn 40}%
\special{pa 200 130}%
\special{pa 1000 730}%
\special{fp}%
\special{pa 1000 730}%
\special{pa 1800 130}%
\special{fp}%
\special{pa 1050 700}%
\special{pa 1050 1500}%
\special{fp}%
\special{pa 950 700}%
\special{pa 950 1500}%
\special{fp}%
\end{picture}}%

%% file: figure/web-0-11-0.tex
\unitlength 0.1in
\begin{picture}( 14.0000, 14.0000)(  3.0000,-21.0000)
\special{pn 40}%
\special{ar 1000 1200 700 700  0.0000000  6.2831853}%
\end{picture}%

%% file: figure/web-1-11-0.tex
\unitlength 0.1in
\begin{picture}( 14.0000, 14.0000)(  3.0000,-22.0000)
\special{pn 40}%
\special{ar 1000 1000 700 400  0.0000000  6.2831853}%
\special{pn 40}%
\special{pa 1000 2000}%
\special{pa 1000 1400}%
\special{fp}%
\end{picture}%

%% file: figure/web-2-11-0.tex
\unitlength 0.1in
\begin{picture}( 14.0000, 14.0000)(  3.0000,-22.0000)
\special{pn 40}%
\special{ar 1000 1000 700 400  0.0000000  6.2831853}%
\special{pn 40}%
\special{pa 1050 2000}%
\special{pa 1050 1390}%
\special{fp}%
\special{pa 950 2000}%
\special{pa 950 1390}%
\special{fp}%
\end{picture}%

%% file: figure/web-1-11-1.tex
\unitlength 0.1in
\begin{picture}( 12.0000, 15.0000)(  4.0000,-22.0000)
\special{pn 40}%
\special{pa 1000 500}%
\special{pa 1000 800}%
\special{fp}%
\special{pn 40}%
\special{pa 1000 2000}%
\special{pa 1000 1700}%
\special{fp}%
\special{pn 40}%
\special{ar 1000 1250 600 450  0.0000000  6.2831853}%
\end{picture}%

%% file: figure/web-11-1.tex
{\unitlength 0.1in%
\begin{picture}(16.0000,16.0000)(2.0000,-22.0000)%
\special{pn 40}%
\special{pa 1000 500}%
\special{pa 1000 1400}%
\special{fp}%
\special{pa 1000 1400}%
\special{pa 200 2000}%
\special{fp}%
\special{pa 1000 1400}%
\special{pa 1800 2000}%
\special{fp}%
\end{picture}}%

%% file: figure/web-0-22-0.tex
\unitlength 0.1in
\begin{picture}( 14.0000, 14.0000)(  3.0000,-21.0000)
\special{pn 40}%
\special{ar 1000 1200 700 700  0.0000000  6.2831853}%
\special{ar 1000 1200 600 600  0.0000000  6.2831853}%
\end{picture}%

%% file: figure/web-2-11-1.tex
\unitlength 0.1in
\begin{picture}( 12.0000, 15.0000)(  4.0000,-22.0000)
\special{pn 40}%
\special{pa 1000 500}%
\special{pa 1000 800}%
\special{fp}%
\special{pn 40}%
\special{pa 1050 2000}%
\special{pa 1050 1690}%
\special{fp}%
\special{pa 950 2000}%
\special{pa 950 1690}%
\special{fp}%
\special{pn 40}%
\special{ar 1000 1250 600 450  0.0000000  6.2831853}%
\end{picture}%

%% file: figure/web-11-2.tex
{\unitlength 0.1in%
\begin{picture}(16.0000,16.0000)(2.0000,-22.0000)%
\special{pn 40}%
\special{pa 1050 500}%
\special{pa 1050 1400}%
\special{fp}%
\special{pa 950 500}%
\special{pa 950 1400}%
\special{fp}%
\special{pa 1000 1370}%
\special{pa 200 1970}%
\special{fp}%
\special{pa 1000 1370}%
\special{pa 1800 1970}%
\special{fp}%
\end{picture}}%

%% file: figure/web-11-0-22.tex
{\unitlength 0.1in%
\begin{picture}(16.0000,16.0000)(2.0000,-20.0000)%
\special{pn 40}%
\special{ar 1000 2400 1000 1000 3.7850938 5.6396842}%
\special{pn 40}%
\special{ar 1000 -400 1000 1000 0.6435011 2.4980915}%
\special{pn 40}%
\special{ar 1000 -400 900 900 0.6435011 2.4980915}%
\end{picture}}%

%% file: figure/web-r2times1-1p.tex
\unitlength 0.1in
\begin{picture}( 13.0000, 15.0000)(  2.0000,-18.0000)
\special{pn 40}%
\special{pa 100 100}%
\special{pa 650 650}%
\special{fp}%
\special{pa 1050 1050}%
\special{pa 1600 1600}%
\special{fp}%
\special{pa 100 1600}%
\special{pa 1600 100}%
\special{fp}%
\put(1.000,-12.00){\makebox(0,0){\Huge $2\varpi_1$}}%
\end{picture}%

%% file: figure/web-r2times1-2times1p.tex
\unitlength 0.1in
\begin{picture}( 13.0000, 15.0000)(  2.0000,-18.0000)
\special{pn 40}%
\special{pa 100 100}%
\special{pa 650 650}%
\special{fp}%
\special{pa 1050 1050}%
\special{pa 1600 1600}%
\special{fp}%
\special{pa 100 1600}%
\special{pa 1600 100}%
\special{fp}%
\put(1.000,-12.00){\makebox(0,0){\Huge $2\varpi_1$}}%
\put(16.000,-12.00){\makebox(0,0){\Huge $2\varpi_1$}}%
\end{picture}%

%% file: figure/line.tex
\unitlength 0.1in
\begin{picture}(  2.0000,  2.0000)(  1.0000, -4.50000)
\special{pn 8}%
\special{pa 200 400}%
\special{pa 200 200}%
\special{fp}%
\end{picture}%

%% file: figure/d-line.tex
\unitlength 0.1in
\begin{picture}(  2.0000,  2.0000)(  1.0000, -4.50000)
\special{pn 8}%
\special{pa 200 400}%
\special{pa 200 200}%
\special{fp}%
\special{pn 8}%
\special{pa 220 400}%
\special{pa 220 200}%
\special{fp}%
\end{picture}%